\documentclass[12pt]{article}
\usepackage{}

\usepackage{epsfig}
\usepackage{latexsym}
\usepackage{caption}
\usepackage{amsfonts}
\usepackage{amssymb}
\usepackage{mathrsfs}
\usepackage{amsmath}
\usepackage{enumerate}
\usepackage{graphics}
\usepackage{MnSymbol}
\usepackage{float}
\usepackage{pict2e}
\usepackage{subfigure}

\usepackage{color}

\makeatletter

\newcommand{\Rmnum}[1]{\expandafter\@slowromancap\romannumeral #1@}

\makeatother

\begin{document}

\newtheorem{theorem}{Theorem}
\newtheorem{observation}[theorem]{Observation}
\newtheorem{corollary}[theorem]{Corollary}
\newtheorem{algorithm}[theorem]{Algorithm}
\newtheorem{definition}[theorem]{Definition}
\newtheorem{guess}[theorem]{Conjecture}
\newtheorem{claim}[theorem]{Claim}
\newtheorem{problem}[theorem]{Problem}
\newtheorem{question}[theorem]{Question}
\newtheorem{lemma}[theorem]{Lemma}
\newtheorem{proposition}[theorem]{Proposition}
\newtheorem{fact}[theorem]{Fact}

\makeatletter
  \newcommand\figcaption{\def\@captype{figure}\caption}
  \newcommand\tabcaption{\def\@captype{table}\caption}
\makeatother

\newtheorem{acknowledgement}[theorem]{Acknowledgement}

\newtheorem{axiom}[theorem]{Axiom}
\newtheorem{case}[theorem]{Case}
\newtheorem{conclusion}[theorem]{Conclusion}
\newtheorem{condition}[theorem]{Condition}
\newtheorem{conjecture}[theorem]{Conjecture}
\newtheorem{criterion}[theorem]{Criterion}
\newtheorem{example}[theorem]{Example}
\newtheorem{exercise}[theorem]{Exercise}
\newtheorem{notation}{Notation}
\newtheorem{solution}[theorem]{Solution}
\newtheorem{summary}[theorem]{Summary}

\newenvironment{proof}{\noindent {\bf
Proof.}}{\rule{3mm}{3mm}\par\medskip}
\newcommand{\remark}{\medskip\par\noindent {\bf Remark.~~}}
\newcommand{\pp}{{\it p.}}
\newcommand{\de}{\em}
\newcommand{\mad}{\rm mad}
\newcommand{\qf}{Q({\cal F},s)}
\newcommand{\qff}{Q({\cal F}',s)}
\newcommand{\qfff}{Q({\cal F}'',s)}
\newcommand{\f}{{\cal F}}
\newcommand{\ff}{{\cal F}'}
\newcommand{\fff}{{\cal F}''}
\newcommand{\fs}{{\cal F},s}
\newcommand{\s}{\mathcal{S}}
\newcommand{\G}{\Gamma}
\newcommand{\g}{(G_3, L_{f_3})}
\newcommand{\wrt}{with respect to }
\newcommand {\nk}{ Nim$_{\rm{k}} $  }

\newcommand{\q}{\uppercase\expandafter{\romannumeral1}}
\newcommand{\qq}{\uppercase\expandafter{\romannumeral2}}
\newcommand{\qqq}{\uppercase\expandafter{\romannumeral3}}
\newcommand{\qqqq}{\uppercase\expandafter{\romannumeral4}}
\newcommand{\qqqqq}{\uppercase\expandafter{\romannumeral5}}
\newcommand{\qqqqqq}{\uppercase\expandafter{\romannumeral6}}

\newcommand{\qed}{\hfill\rule{0.5em}{0.809em}}

\newcommand{\var}{\vartriangle}

\title{{\large \bf  Multiple list colouring triangle free planar graphs}}

\author{   Yiting Jiang \and Xuding Zhu\thanks{Department of Mathematics, Zhejiang Normal University,  China.  E-mail: xudingzhu@gmail.com. Grant Number: NSFC 11571319, ZJNSF LD19A010001.}}

\maketitle

\begin{abstract}
	This paper proves that for each positive integer $m$, there is a triangle free planar graph $G$ which is not $(3m+ \lceil \frac m{17} \rceil-1, m)$-choosable.

\bigskip
	
\end{abstract}

\section{Introduction}

Colouring of triangle free planar graphs has been studied extensively in the literature. It was proved by Gr\"{o}tzsch \cite{Grotzsch} that every triangle free planar graph is $3$-colourable. On the other hand, Voigt \cite{n3MV}  showed that there are triangle free planar graphs that are not $3$-choosable. 
In this paper, we are interested in multiple list colouring of triangle free planar graphs. 

A $b$-fold colouring of a graph   $G$ is a mapping $\phi$ which assigns to each vertex $v$ of $G$ a set $\phi (v)$ of $b$ colours, so that adjacent vertices receive disjoint colour sets. An $(a,b)$-colourig of $G$ is a $b$-fold colouring $\phi$ of $G$ such that $\phi(v) \subseteq \{1,2,\cdots,a\}$ for each vertex $v$. The {\em fractional chromatic number} of $G$ is 
$$\chi_f(G)=inf\{\frac a b :G \ is \ (a,b)-colourable\}.$$
An $a$-list assignment of $G$ is a mapping $L$ which assigns to each vertex $v$ a set $L(v)$ of $a$ permissible colours. A $b$-fold $L$-colouring of $G$ is a $b$-fold colouring $\phi$ of $G$ such that $\phi(v)\subseteq L(v)$ for each vertex $v$. We say $G$ is $(a,b)$-choosable if for any $a$-list assignment $L$ of $G$, there is a $b$-fold $L$-colouring of $G$.
The {\em choice number} $ch(G)$ of $G$ is 
$$ch(G)=\min \{  a:G \ is \ (a,1)-choosable\}.$$
  The {\em fractional choice number} of $G$ is $$ch_f(G)=inf\{\frac a b :G \ is \ (a,b)-choosable\}.$$
It was proved by Alon, Tuza and Voigt \cite{Chf} that for any finite graph $G$, $\chi_f(G)=ch_f(G)$ and moreover the infimum in the definition of $ch_f(G)$ is attained and hence can be replaced by minimum. 
This implies that if $G$ is $(a,b)$-colourable, then for some interger $m$, $G$ is $(am,bm)$-choosable. 

Recently, it was shown by Dvo\v{r}\'{a}k,   Sereni and   Volec \cite{DSV} that an $n$-vertex triangle free planar graph $G$ has fractional chromatic number at most $\frac{9n}{3n+1}$. 
Therefore, for each  triangle free planar graph $G$, there is an integer $m$ such that $G$ is $(3m-1,m)$-choosable.  

A natural question is whether there is a constant $m$ such that every triangle free planar graph $G$ is $(3m,m)$-choosable. If not, what is the smallest real number $\epsilon$ such that every triangle free planar graph $G$ is $(\lceil (3+\epsilon) m \rceil,m)$-choosable.

For a positive real number $\alpha$, we say a graph $G$ is {\em strongly fractional $\alpha$-choosable} if for any positive integer $m$, $G$ is $(\lceil \alpha m \rceil, m)$-choosable. We define the {\em strong fractional choice number} $ch_f^*(G)$ of $G$ as 
$$ch_f^*(G) = \inf \{\alpha: \text{$G$ is strongly fractional $\alpha$-choosable}\}.$$
The strong fractional choice number of a family ${\cal G}$ of graphs is defined as 
$$ch_f^*({\cal G}) = \sup \{ch_f^*(G): G \in {\cal G}\}.$$  
We are interested in the strong fractional choice number of the class of triangle free planar graphs. 

The strong fractional choice number of planar graphs was studied in \cite{Zhu}, where it was shown that for each positive integer $m$, there is a planar graph $G$ which is not $(4m+ \lfloor  \frac {2m-1}{9} \rfloor, m)$-choosable. Let ${\cal P}$ be the class of planar graphs. Then we have $ch_f^*({\cal P}) \ge 4+ \frac 29.$

In this paper, we prove the following result. Let ${\cal Q}$ be the family of triangle free planar graphs.

\begin{theorem}
	\label{main}
	For each positive integer $m$, there is a triangle free planar graph $G$ which is not $(3m+ \lceil  \frac m{17} \rceil -1, m)$-choosable.
	Consequently, $ch_f^*({\cal  Q}) \ge 3+ \frac 1{17}$.
\end{theorem}

The $m=1$ case of Theorem 1 is equivalent to say that there are non-$3$-choosable triangle free planar graphs, which was proved by Voigt \cite{n3MV}. Some other non-$3$-choosable triangle free planar graphs were constructed by Montassier \cite{n3MM} and by Glebov, Kostochka and Tashkinov \cite{smn3}.  

\section{The proof of Theorem \ref{main}}

In this section, $m$ is an arbitrary but fixed positive integer. We shall construct a triangle free planar graph $G$ which is not $(3m+ \lceil \frac m{17} \rceil-1, m)$-choosable.

\begin{figure}[!htp]
	\begin{center}
		\includegraphics[scale=0.5]{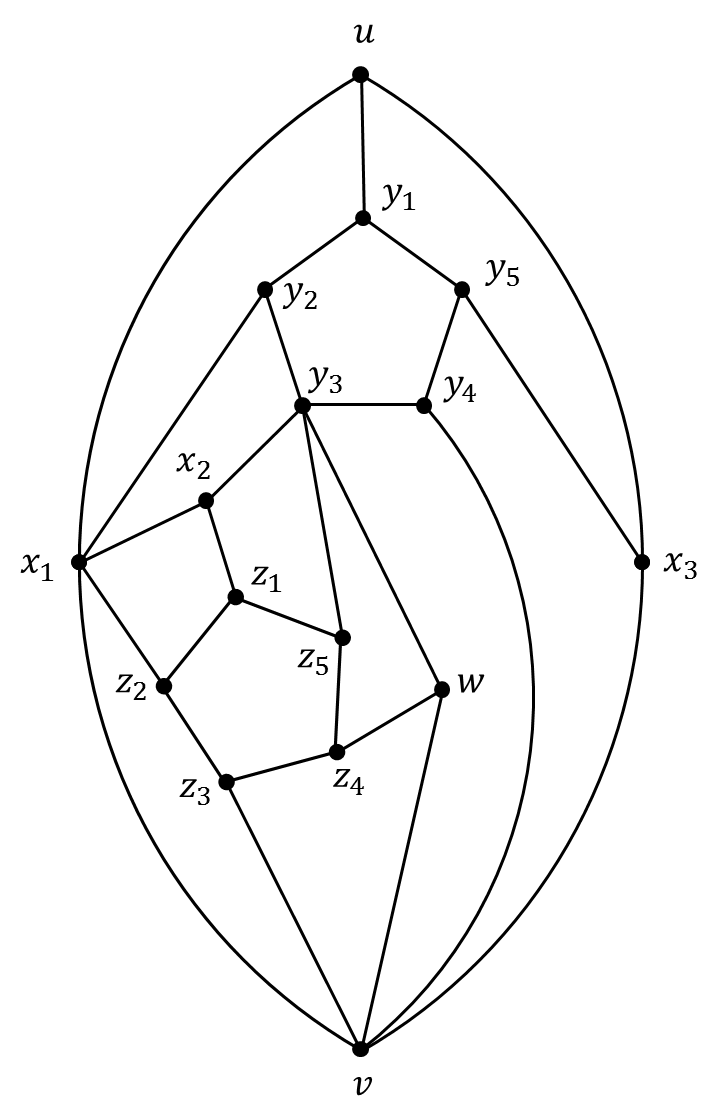}
	\end{center}
	\caption{The graph $H$}
\end{figure}

\begin{lemma} 
	\label{lem1}
	Let $H$ be the graph as shown in Figure 1. Let $\epsilon$ be the real number such that $\epsilon m = \lceil \frac m{17} \rceil-1 $. Let $A, B$ be disjoint sets of $m$ colours. 
	Then there is a list assignment $L$ of $H$ for which the following hold:
 \begin{itemize}
 		\item[1] $L(u)=A$ and $L(v)=B$.
 	\item[2]  $|L(s)|=3m+ \epsilon m$ for each vertex $s \ne u, v$.
 	\item[3]  There is no m-fold $L$-colouring of $H$.
 \end{itemize}	 
\end{lemma}

 \begin{proof} Let $A, B, C, D, E$ be disjoint colour sets, where $|A|=|B|=|C|=m$, $|D|=2m$ and $|E|= \epsilon m$.
 
 Let $L$ be the list assignment of $H$  defined as follows:
 \begin{itemize}
 		\item   $L(u)=A$ and $L(v)=B$.
 	\item $L(x_1)=L(x_2)=L(x_3)=L(w)=A \cup B \cup C \cup E$.
 	\item $L(y_1)=L(y_3)=L(z_5)=A \cup D \cup E$.
 	\item $L(y_4)=L(z_1)=L(z_3)=B \cup D \cup E$.
 	\item $L(y_2)=L(y_5)=L(z_2)=L(z_4)=C \cup D \cup E$.
 \end{itemize}
 
 We shall show that $H$ is not $m$-fold $L$-colourable. 
 Assume to the contrary that $\phi$ is an $m$-fold $L$-colouring of $H$.

The vertices   $y_1,y_2,y_3, y_4, y_5$ induce a $5$-cycle. So each colour in $D \cup E$ can be used at most twice on these five vertices. At least $(1-\epsilon)m$ colours of $C$ are used on vertex $x_3$. Hence at most $\epsilon m$ of colours from $C$ can be used at vertex $y_5$. Similarly, at most  $\epsilon m$ of colours from $C$ can be used at vertex $y_2$. 
 Assume $ \tau m$ colours from $A$ are used at vertex $y_3$.
 Since altogether we use $5m$ colours to colour these five vertices, we conclude
 $\tau+ 4\epsilon+4 \ge 5$. Hence
 $$\tau \ge  1-4\epsilon.$$
 
 As  $m$ colours from $C \cup E$ are used by $x_1$, we know that  at most $\epsilon m$ colours of $C \cup E$ are used at vertex $x_2$. As at least 
 $(1-4\epsilon  )m$ of $A$ is used at $y_3$, it follows that 
 at most $4\epsilon m$ colours of $A$ are used at vertex $x_2$. 
 
 So at least $(1-5\epsilon)m$ colours from $B$ are used at vertex $x_2$. 
 
 At most $\epsilon m$ colours from $C$ are used at vertex $z_2$. At most $5\epsilon   m$ colours from $B$ are used at vertex $z_1$. At most $4\epsilon  m$ colours from $A$ are used at vertex $z_5$. Each colour from $D \cup E$ is used at most twice among vertices $z_1,z_2,z_3,z_4,z_5$. Assume $\sigma m$ colours from $C$ are used at vertex $z_4$. Then $\sigma+12\epsilon  + 4 \ge 5$. Hence $$\sigma \ge 1-12\epsilon.$$
 
 Therefore, at most $4\epsilon m$ colours from $A$ can be used at vertex $w$, at most $12 \epsilon m$ colours from $C$ can be used at vertex $w$ and at most $\epsilon m$ colours from $E$ can be used at vertex $w$. So the total number of colour available to $w$ is at most  $  17 \epsilon m.$  
 Since $\epsilon < \frac 1{17}$, we arrive at a contradiction.
\end{proof}

 It can be verified that for the list assignment $L$ defined above, if $E$ is a set of $\lceil \frac m {17} \rceil$ colours, then there is an $L$-colouring of $H$.

 Let $p=\binom{3m+\epsilon m}{m}$, and let $G$ be obtained from the disjoint union of $p^2$ copies of $H$ by identifying all the copies of $u$ into a single vertex (also named as $u$) and all the copies of $v$ into a single vertex (also named as $v$). It is obvious that $G$ is a triangle free planar graph.
 
 Now we show that $G$ is not $(3m+\epsilon m, m)$-choosable. Let $X$ and $Y$ be two disjoint sets of $3m+\epsilon m$ colours. Let $L(u)=X$ and $L(v)=Y$. There are $p^2$ possible $m$-fold $L$-colourings of   $u$ and $v$. Each such a colouring $\phi $ corresponds to one copy of $H$. In that copy of $H$, define the list assignment as in the proof of Lemma \ref{lem1}, by replacing $A$ with $\phi (u)$ and $B$ with $\phi (v)$. Now Lemma 2 implies that no $m$-fold colouring of $u$ and $v$ can be extended to an $m$-fold $L$-colouring of $G$. This completes the proof of Theorem \ref{main}. 
 
 \section{Some open questions}
 
We define the choice number of a class ${\cal G}$ of graphs as 
$ch({\cal G}) = \max \{ch(G): G \in {\cal G}\}$. Then $ch({\cal Q}) = 4$. The strong fractional choice number of graphs is intended to provide a finer scale for measuring the choosability of graphs.

It was conjectured by Erd\H{o}s,
 Rubin and Taylor \cite{ERT} that if a graph $G$ is $(a, b)$-choosable, then for any positive integer $m$, $G$ is
 $(am, bm)$-choosable. If this conjecture were true, then $ch_f^*(G) \le ch (G)$ for any graph $G$. However, very recently, this conjecture is refuted by Dvo\v{r}\'{a}k,  Hu and   Sereni \cite{DHS} who proved that for any integer $k \ge 4$, there is a $k$-choosable graph which is not $(2k,2)$-choosable. Nevertheless, for a graph $G$ with $ch(G)=k$, to show that  $ch_f^*(G) \le k$, it suffices to show that for any integer $m \ge 1$, $G$ is $(km+1,m)$-choosable. So it is still possible that the inequality $ch_f^*(G) \le ch (G)$ holds for every graph $G$.
 
 As triangle free planar graphs are $3$-degenerate,  
 we know that for any positive integer $m$, every triangle free planar graph $G$ is $(4m,m)$-choosable.
 So $ch_f^*({\cal Q}) \le 4$.  A challenging problem is to determine the value of $ch_f^*({\cal Q})$.

  \begin{question}
  	\label{q0}
  What is the value of $ch_f^*({\cal Q})$?
  \end{question}

 Question \ref{q0} may be very difficult. The following question is easier, but also remains open. 
 
 \begin{question}
 	\label{q1}
 	Is there a positive integer $m$ such that every triangle free planar graph is $(4m-1,m)$-choosable?
 \end{question}
 
 If such an $m$ exists, then the smallest possible value of $m$ is $2$. 
 
 \begin{question}
 	\label{q2}
 	Is it true that every triangle free planar graph is $(7,2)$-choosable?
 \end{question}

 It is known \cite{XZ} that for any finite graph $G$, $ch_f^*(G)$ is a rational number. Which rational numbers are the strong fractional choice number of graphs remains an open question. The question restricted to planar graphs and triangle free planar graphs are also interesting.

 \begin{question}
 	\label{q3}
 	Which rational numbers are the strong fractional choice number of  planar graphs? Which rational numbers are the strong fractional choice number of triangle free planar graphs? 
 \end{question}

\end{document}